\documentclass[11pt,a4paper]{article}

\usepackage{amsmath,amsfonts,amsthm,amssymb}
\usepackage[english]{babel}
\usepackage{graphicx}
\usepackage{array}
\usepackage{tikz}
\usetikzlibrary{arrows}
\usetikzlibrary{positioning}

\usepackage[a4paper, total={6in, 9in}]{geometry}

\newtheorem{theorem}{Theorem}[section]
\newtheorem{lemma}[theorem]{Lemma}
\newtheorem{conjecture}[theorem]{Conjecture}
\newtheorem{corollary}[theorem]{Corollary}
\newtheorem{definition}[theorem]{Definition}

\newtheorem{question}[theorem]{Question}
\newtheorem*{claim*}{Claim}

\begin{document}

\title{Decomposing graphs into a constant number of locally irregular subgraphs}

\author{Julien Bensmail\thanks{The authors were supported by ERC Advanced Grant GRACOL, project no. 320812.}\hspace{2cm} Martin Merker\footnotemark[1]\hspace{2cm} Carsten Thomassen\footnotemark[1]\\~\\
			 Department of Applied Mathematics and Computer Science \\  Technical University of Denmark \\  DK-2800 Lyngby, Denmark}

\date{\today}

\maketitle

\begin{abstract}
A graph is  \textit{locally irregular} if no two adjacent vertices have the same degree. The  \textit{irregular chromatic index} $\chi_{\rm irr}'(G)$ of a graph $G$ is the smallest number of locally irregular subgraphs needed to edge-decompose $G$. Not all graphs have such a decomposition, but Baudon, Bensmail, Przyby{\l}o, and Wo\'zniak conjectured that if $G$ can be decomposed into locally irregular subgraphs, then $\chi_{\rm irr}'(G)\leq 3$. 
In support of this conjecture, Przyby{\l}o showed that $\chi_{\rm irr}'(G)\leq 3$ holds whenever $G$ has minimum degree at least $10^{10}$. 

Here we prove that every bipartite graph $G$ which is not an odd length path satisfies $\chi_{\rm irr}'(G)\leq 10$. This is the first general constant upper bound on the irregular chromatic index of bipartite graphs.
Combining this result with Przyby{\l}o's result, we show that $\chi_{\rm irr}'(G) \leq 328$ for every graph $G$ which admits a decomposition into locally irregular subgraphs.
Finally, we show that $\chi_{\rm irr}'(G)\leq 2$ for every $16$-edge-connected bipartite graph $G$.
\end{abstract}

\section{Introduction}

A graph $G$ is \textit{locally irregular} if any two adjacent vertices have distinct degrees.
This concept of local irregularity was investigated by Karo\'nski, {\L}uczak, and Thomason in~\cite{KLT04}, where they introduced the so-called 1-2-3 Conjecture.
An edge-weighting of $G$ is called \textit{neighbour-sum-distinguishing}, if for every two adjacent vertices of $G$ the sums of incident weights are distinct.
The least number $k$ for which $G$ admits a neighbour-sum-distinguishing edge-weighting using weights $1,2,\ldots ,k$
is denoted $\chi'_\Sigma(G)$.

\begin{conjecture}[1-2-3 Conjecture~\cite{KLT04}]
For every graph $G$ with no component isomorphic to $K_2$, we have $\chi'_\Sigma(G) \leq 3$.
\end{conjecture}

This conjecture is equivalent to stating that a graph can be made locally irregular by replacing some of its edges by two or three parallel edges. 
Although the 1-2-3 Conjecture has received considerable attention in the last decade, it is still an open question.
The best result so far was shown by Kalkowski, Karo\'nski, and Pfender~\cite{KKP10} who proved
$\chi'_\Sigma(G) \leq 5$ whenever $G$ has no component isomorphic to $K_2$.
For more details, we refer the reader to the survey by Seamone~\cite{Sea12} on the 1-2-3 Conjecture and related problems.

A different approach was taken by Baudon, Bensmail, Przyby{\l}o, and Wo\'{z}niak in~\cite{BBPW15}, where edge-decompositions of graphs into locally irregular graphs were studied. From now on, all graphs we consider are simple and finite.
A decomposition into locally irregular subgraphs can be regarded as an improper edge-colouring where each colour class induces a locally irregular graph.
We call such an edge-colouring \textit{locally irregular}. 
If $G$ admits a locally irregular edge-colouring, then we call $G$ \textit{decomposable}.
For every decomposable graph $G$, we define the \textit{irregular chromatic index of~$G$}, denoted by $\chi'_{\rm irr}(G)$, as the least number of colours in a locally irregular edge-colouring of $G$.
If $G$ is not decomposable, then $\chi'_{\rm irr}(G)$ is not defined and we call $G$ \textit{exceptional}. 
The following conjecture has a similar flavour to the 1-2-3 Conjecture.

\begin{conjecture}[\cite{BBPW15}] \label{conj-3}
For every decomposable graph $G$, we have $\chi'_{\rm irr}(G) \leq 3$.
\end{conjecture}

Every connected graph of even size can be decomposed into paths of length 2 and is thus decomposable. Hence, all exceptional graphs have odd size and a complete characterisation of exceptional graphs was given in~\cite{BBPW15}. 
To state this characterisation, we first need to define a family $\mathcal{T}$ of graphs. A connected graph $G$ belongs to $\mathcal{T}$ if and only if $G$ has a nonempty collection of triangles, $G$ has no other cycles, $G$ has maximum degree at most $3$, all vertices not in a triangle have degree at most $2$, and if $P$ is path in $G$ whose intermediate vertices all have degree $2$ in $G$ and $P$ is maximal with this property, then $P$ has odd length if and only if both ends are in triangles. 

\begin{theorem}[\cite{BBPW15}] \label{charac-exceptions}
A connected graph is exceptional, if and only if it is (1) a path of odd length, (2) a cycle of odd length, or (3) a member of $\mathcal{T}$.
\end{theorem}

The number~$3$ in Conjecture~\ref{conj-3} cannot be decreased to~$2$, as shown for example by cycles with length congruent to~$2$ modulo~$4$ and complete graphs.
In~\cite{BBPW15}, Conjecture~\ref{conj-3} was verified for several classes of graphs such as trees, complete graphs, and regular graphs with degree at least~$10^7$.
In~\cite{BBS15}, Baudon, Bensmail, and Sopena showed that determining the irregular chromatic index of a graph is \textsf{NP}-complete in general,
and that, although infinitely many trees have irregular chromatic index~$3$, the same problem for trees can be solved in linear time.
More recently, Przyby{\l}o~\cite{Prz15} gave further evidence for Conjecture~\ref{conj-3} by verifying it for graphs of large minimum degree.

\begin{theorem}[\cite{Prz15}] \label{theorem:przybylo}
For every graph $G$ with minimum degree at least $10^{10}$, we have $\chi'_{\rm irr}(G) \leq 3$.
\end{theorem}

Despite this result, Conjecture~\ref{conj-3} is still wide open, even in much weaker forms.
Until now it was not known whether there exists a constant~$c$ such that $\chi'_{\rm irr}(G) \leq c$ holds for every decomposable graph $G$. 
This was also an open problem when restricted to bipartite graphs, see~\cite{BBPW15,BBS15,BS16, Prz15}.

In this paper we show $\chi'_{\rm irr}(G) \leq 328$ for every decomposable graph $G$, hence providing the first constant upper bound on the irregular chromatic index.
The proof consists of three steps. First, we show in Section 2 that we can restrict our attention to connected graphs of even size. Notice that every connected graph of even size can be decomposed into paths of length 2 and is thus decomposable. We show that every decomposable graph $G$ of odd size contains a locally irregular subgraph $H$ such that all connected components of $G-E(H)$ have even size. In Section 3, we investigate bipartite graphs $G$ of even size and show that $\chi'_{\rm irr}(G) \leq 9$ holds in this case. Finally, in Section 4 we decompose a connected graph $G$ of even size into a graph $H$ of minimum degree $10^{10}$ and a $(2\cdot 10^{10})$-degenerate graph $D$ in which every component has even size. We use Theorem~\ref{theorem:przybylo} to decompose $H$, and we further decompose $D$ into $36$ bipartite graphs of even size. By using our result for bipartite graphs, this results in a decomposition of $G$ into $3+9\cdot 36 = 327$ locally irregular subgraphs. 

We have proved a constant upper bound on the irregular chromatic index, but there is still a significant gap between our result and the conjectured value $3$. 
We conclude this article by showing in Section 5 that $\chi'_{\rm irr}(G) \leq 2$ for every $16$-edge-connected bipartite graph~$G$.

\section{Reduction to graphs of even size} \label{section:even-size}

In this section we show that the weakening of Conjecture~\ref{conj-3}, where $3$ is replaced by a larger number, can be reduced to connected graphs of even size.
More precisely, given a decomposable graph $G$ with odd size, we can always remove a locally irregular subgraph $H$ from $G$, so that all connected components of $G - E(H)$ have even size.

\begin{lemma}\label{lemma-oneedge}
Let $G$ be a connected graph of odd size. For every vertex $v\in V(G)$ there exists an edge $e$ incident with $v$ such that every connected component of $G-e$ has even size.
\end{lemma}

\begin{proof}
Let $E(v)$ denote the set of edges incident with $v$. If $e\in E(v)$ is not a cut-edge, then $G-e$ is connected and of even size. We may thus assume that all edges in $E(v)$ are cut-edges. For every $e\in E(v)$, let $H_e$ denote the connected component of $G-e$ not containing $v$. Now
$$ E(G) = \bigcup_{e\in E(v)} E(H_e) \cup \{e\}\,.$$
Since $|E(G)|$ is odd, there exists $e\in E(v)$ for which $|E(H_e) \cup \{e\}|$ is odd. Thus, $H_e$ is of even size, and so is the other connected component of $G-e$.
\end{proof}

\begin{lemma}\label{lemma-pathlength2}
Let $G$ be a connected graph of even size. For every vertex $v\in V(G)$ there exists a path $P$ of length 2 containing $v$ such that every connected component of $G-E(P)$ has even size.
\end{lemma}

\begin{proof}
Let $e$ be an edge incident with $v$. Then $G-e$ has precisely one connected component of odd size, and $e$ is incident with a vertex $u$ of that component, possibly $u=v$. By Lemma~\ref{lemma-oneedge} we can delete an edge $f$ incident with $u$ so that every component of $G-\{e,f\}$ has even size. Since $e$ and $f$ are incident, they form a path $P$ of length 2.
\end{proof}

\begin{theorem}\label{thm-odd}
Let $G$ be a connected graph of odd size. If $G$ is decomposable, then $G$ contains a locally irregular subgraph $H$ such that every connected component of $G-E(H)$ has even size.
\end{theorem}
\begin{proof}
We show that we can choose $H$ to be isomorphic to $K_{1,3}$ or to $K_{1,3}$ where two edges are subdivided once. 
Assume that $G$ is a graph for which we cannot delete one of these two graphs such that every connected component in the resulting graph is of even size. 
If $G$ has maximum degree at most $2$ and odd size, then $G$ is exceptional. We can thus assume that $G$ has maximum degree at least 3.
Notice that every vertex $v$ of degree at least $3$ in $G$ must be a cutvertex, since otherwise we can delete a claw (i.e. a subgraph isomorphic to $K_{1,3}$) centred at $v$.

First, suppose that $G$ contains a cycle $C$. Let $V_C$ denote the vertices of $C$ with degree at least $3$.
For every $v\in V_C$, let $E_C(v)$ denote the two edges of $C$ that are incident with $v$. If $G-E_C(v)$ is connected, then we can use Lemma~\ref{lemma-oneedge} to delete one more edge at $v$ so that every connected component in the resulting graph has even size. We may thus assume that $G-E_C(v)$ is disconnected. Let $G_C(v)$ denote the connected component of $G-E_C(v)$ containing $v$. If $|E(G_C(v))|$ is odd, then we can again use Lemma~\ref{lemma-oneedge} to delete one more edge at $v$ to reach the desired conclusion. Thus we may assume that $|E(G_C(v))|$ is even for all $v\in V_C$.
By Lemma~\ref{lemma-pathlength2}, there exists a path $P_v$ of length 2 in $G_C(v)$ incident with $v$ such that every connected component of $G_C(v)-E(P_v)$
has even size. If $v$ is the middle vertex of $P_v$, then $P_v$ together with one of the two edges in $E_C(v)$ forms a claw whose removal leaves a graph where every connected component has even size. Thus, we may assume that $v$ is an endvertex of $P_v$. If $C$ has length at least 4, then let $P_C$ be a path of length 3 in $C$ in which $v$ has degree 2. The graph $P_v\cup P_C$ is locally irregular and it is easy to see that every connected component of $G-E(P_v)-E(P_C)$ has even size.

Thus we may assume that all cycles of $G$ have length 3. Suppose two cycles $C_1,C_2$ have a vertex $v$ in common. Choose an edge $e_i$ incident with $v$ in $C_i$ for $i\in \{1,2\}$. Now $G-\{e_1,e_2\}$ is connected, so we can apply Lemma~\ref{lemma-oneedge} to delete one more edge at $v$ so that every connected component has even size.

So far, we have shown that triangles are the only cycles in $G$ and that any two triangles are disjoint. Now we show that there exists no induced claw in $G$. Suppose for a contradiction that $v$ is a vertex of degree at least 3 which is a center of a claw. If $v$ is contained in a triangle, then we assume that the degree of $v$ is at least 4. Since any two triangles are disjoint, there exists at most one edge between the neighbours of $v$.
By Lemma~\ref{lemma-oneedge}, we can delete an edge $uv$ so that every component of $G'=G-uv$ has even size. By our choice of $v$, there exists two neighbours $u_1$ and $u_2$ of $v$ such that $\{u,u_1,u_2\}$ is an independent set in $G$. Let $G_1$ denote the connected component of odd size in $G'-u_1v$. If $G_1$ contains $v$, then we can delete a third edge $e$ at $v$ by Lemma~\ref{lemma-oneedge} such that all components of $G'-u_1v-e$ have even size. Thus, we can assume that $G_1$ contains $u_1$ but not $v$. Similarly, we may assume that the odd component $G_2$ of $G'-u_2v$ contains $u_2$ but not $v$. Now we can apply Lemma~\ref{lemma-oneedge} to delete an edge $e_i$ incident with $u_i$ in $G_i$ such that every connected component of $G_i-e_i$ has even size for $i\in \{1,2\}$. Thus, every connected component of $G'-e_1-e_2-u_1v-u_2v$ has even size.
Since $G_1,G_2$ are distinct components of $G-v$, the graph we removed is isomorphic to $K_{1,3}$ where two edges are subdivided once. This contradicts our choice of $G$, implying that $G$ has no induced claw.

Thus we may assume that the maximum degree in $G$ is 3 and that every vertex of degree 3 is contained in a triangle. Since there are no other cycles, this implies that the contraction of all triangles results in a tree of maximum degree 3. All that remains to show is that the parities of the path lengths are the same as for the exceptional graphs. Let $P$ be a path joining a leaf in $G$ with a triangle $C$. Let $v$ be the common vertex of $P$ and $C$. Now $P=G_C(v)$ and since $|E(G_C(v))|$ is even, the length of $P$ is even.
Finally, let $P$ be a path joining two different triangles $C_1$ and $C_2$. If $v_1$ and $v_2$ denote the endvertices, then $$|E(G)|=|E(G_{C_1}(v_1))|+|E(G_{C_2}(v_2))| - |E(P)|\,.$$ Since $|E(G_{C_1}(v_1))|$ and $|E(G_{C_2}(v_2))|$ are even and $|E(G)|$ is odd, we get that $|E(P)|$ must also be odd. This shows that $G$ is exceptional.
\end{proof}

\section{Locally irregular decompositions of bipartite graphs} \label{section:bipartite}

We now focus on the irregular chromatic index of bipartite graphs.
Recall that the only bipartite exceptional graphs are odd length paths.
In Corollary~\ref{corollary:bipodd} we show $\chi'_{\rm irr}(G) \leq 10$ for every decomposable bipartite graph $G$, which is the first constant upper bound on $\chi'_{\rm irr}$ for bipartite graphs.

If all vertices in one partition class of the bipartite graph $G$ have even degree, while the vertices in the other partition class have odd degree, then $G$ is locally irregular. The idea of the proof is to remove some well-behaved subgraphs from $G$ to obtain a graph which is very close to this structure. These well-behaved subgraphs include a particular kind of forest, which is defined as follows.

\begin{definition}
We say a forest is \textbf{balanced} if it has a bipartition such that all vertices in one of the partition classes have even degree.
\end{definition}

Since a balanced forest cannot contain an odd length path as a connected component,
it follows from~\cite{BBPW15} that $\chi'_{\rm irr}(F) \leq 3$ for every balanced forest $F$.
The characterisation of trees $T$ with $\chi'_{\rm irr}(T) \leq 2$ in~\cite{BBS15} implies that even $\chi'_{\rm irr}(F) \leq 2$ holds for balanced forests $F$.
For the sake of completeness, we present a short proof of this special case.

\begin{lemma}\label{lemma-balanced}
If $F$ is a balanced forest, then $F$ has a $2$-colouring of the edges, such that each colour induces a locally irregular graph, and, for each vertex $v$ in the partition class with no odd degree vertex, all edges incident with $v$ have the same colour. In particular, $\chi_{\rm irr}'(F)\leq 2$.
\end{lemma}
\begin{proof}

The proof is by induction on the number of edges of $F$. Clearly, we may assume that $F$ is connected.
Let $A$ and $B$ be the partition classes of $F$, where all vertices in $B$ have even degree. We may assume that some vertex in $A$ has even degree since otherwise we can give all edges of $F$ the same colour. Let $v$ be a vertex in $A$ of even degree $q$. We delete $v$ but keep the edges incident with $v$ and let them go to $q$ new vertices $v_1,v_2, \ldots ,v_q$ each of degree $1$. In other words, we split $F$ into $q$ new trees $T_1,T_2, \ldots ,T_q$ such that the union of their edges is the edge set of $F$. Each of the trees $T_1,T_2, \ldots ,T_q$ is balanced and has therefore a colouring of its edges in colours red and blue satisfying the conclusion of Lemma \ref{lemma-balanced}. This also gives a colouring of the edges of $F$ in colours red and blue. By switchings colours in some of the $T_i$, if necessary, we can ensure that the red degree of $v$ is $1$. This shows that also $F$ satisfies the conclusion of Lemma \ref{lemma-balanced}.
\end{proof}

Apart from balanced forests we shall also delete a subgraph which is the union of a path and an induced cycle. The following lemma gives an upper bound on the irregular chromatic index in this case.

\begin{lemma}\label{lemma-path+cycle}
Let $G$ be a bipartite graph and let $v$ be a vertex in $G$. If $G$ is the edge-disjoint union of an induced cycle $C$ through $v$ and a path $P$ starting at $v$, then $\chi_{\rm irr}'(G)\leq 4$. 
\end{lemma}
\begin{proof}
If the length of $P$ is 0, then $\chi_{\rm irr}'(G)\leq 3$, so we may assume $P$ has positive length.
First suppose that $P$ has odd length. Let $e$ denote the edge of $P$ incident with $v$. It is easy to see that $\chi_{\rm irr}'(C+e)\leq 2$. Thus, $$\chi_{\rm irr}'(G)\leq \chi_{\rm irr}'(C+e) + \chi_{\rm irr}'(P-e) \leq 2+ 2 = 4\,.$$
Now suppose the length of $P$ is even. If the length of $C$ is divisible by $4$, then $$\chi_{\rm irr}'(G)\leq \chi_{\rm irr}'(C) + \chi_{\rm irr}'(P) \leq 2+2 = 4\,.$$ We may therefore assume the length of $C$ is congruent to $2$ modulo $4$. Let $e$ denote the edge of $P$ incident with $v$, and let $f$ denote the edge incident with $e$ on $P$. It is easy to check that if $e$, $f$ and all edges of $C$ incident to $e$ or $f$ are coloured $1$, then this colouring can be extended to a locally irregular $\{1,2\}$-edge-colouring of $C+e+f$. Thus, we have $$\chi_{\rm irr}'(G)\leq \chi_{\rm irr}'(C+e+f) + \chi_{\rm irr}'(P-e-f) \leq 2+ 2 = 4\,.$$
\end{proof}

The following lemma is well-known.

\begin{lemma}\label{lemma-pathsystem}
Let $G$ be a connected graph and let $S$ be a set of vertices. If $S$ is even, then there exists a collection of $\frac{|S|}{2}$ edge-disjoint paths in $G$ such that each vertex in $S$ is an endvertex of precisely one of them.
\end{lemma}

\begin{proof}
Take a collection of paths having the vertices in $S$ as endvertices for which the total length is minimal.
\end{proof}

\begin{corollary}\label{lemma-makeAeven}
If $G$ is a connected bipartite graph of even size with partition classes $A$ and $B$, then there exists a balanced forest $F$ with leaves in $A$ such that in $G-E(F)$ all vertices in $A$ have even degree.
\end{corollary}

\begin{proof}
Notice that since $G$ has even size, the number of vertices in $A$ with odd degree is even. The statement follows by choosing $S$ to be the set of odd-degree vertices in $A$, and $F$ as the union of the paths given by Lemma~\ref{lemma-pathsystem}.
\end{proof}

\begin{corollary}\label{lemma-Balmostodd}
Let $G$ be a connected bipartite graph with partition classes $A$ and $B$, and let $v$ be a vertex in $B$. If all vertices in $A$ have even degree, then there exists a balanced forest $F$ with leaves in $B$ such that in $G-E(F)$ all vertices in $B\setminus \{v\}$ have odd degree.
\end{corollary}

\begin{proof}
Choose $S$ as the set of even-degree vertices in $B$. If $|S|$ is odd, then we apply Lemma~\ref{lemma-pathsystem} to the set $S\cup \{v\}$ or $S \setminus \{v\}$, and if $|S|$ is even we apply Lemma~\ref{lemma-pathsystem} to the set $S$. The union of the paths is the desired balanced forest.
\end{proof}

\begin{lemma}\label{lemma-deletePath}
Let $G$ be a bipartite graph with partition classes $A$ and $B$, and let $v$ be a vertex in $B$. If all vertices in $A$ have even degree and all vertices in $B\setminus \{v\}$ have odd degree, then there exists a path $P$ starting in $v$ such that $G-E(P)$ is locally irregular.
\end{lemma}
\begin{proof}
If $v$ has odd degree, then we can choose $P$ as a path of length 0. If $v$ has even degree and $G$ is not locally irregular, then $v$ is adjacent to a vertex $u_1$ of the same degree. We choose the edge $vu_1$ as the first edge of $P$ and define $G_1=G-vu_1$. If $G_1$ is not locally irregular, then $u_1$ is adjacent to a neighbour $u_2$ of the same degree. In this case we extend $P$ by the edge $u_1u_2$ and define $G_2=G_1-u_1u_2$. We continue like this, defining $G_{i+1}$ if $G_i$ is not locally irregular by deleting a conflict edge $u_iu_{i+1}$. We claim that this process stops with a locally irregular graph $G_k$ and that the deleted edges form a path. Notice that if $G_i$ is not locally irregular, then $u_i$ is incident to a vertex $u_{i+1}$ of the same degree. Moreover, the degree of $u_i$ in $G_i$ is $d(v)-i$, so the degrees $d(u_i)$ form a decreasing sequence. In particular, $u_i\neq u_j$ for $i\neq j$ and $u_i\neq v$ for all $i$. Thus, eventually the process stops with a locally irregular graph $G_k$ and $G-E(G_k)$ is a path of length $k$.
\end{proof}

\begin{lemma}\label{lemma-Aiseven}
Let $G$ be a bipartite graph with partition classes $A$ and $B$. If all vertices in $A$ have even degree, then $\chi_{\rm irr}'(G)\leq 7$.
\end{lemma}
\begin{proof}
We may assume that $G$ is connected. By Lemma~\ref{lemma-Balmostodd}, we can delete a balanced forest $F$ with leaves in $B$ such that in the resulting graph $G'$ there is at most one vertex of even degree in $B$, say $v$. If $v$ does not exist or if $v$ is an isolated vertex in $G'$, then $G'$ is locally irregular and $\chi_{\rm irr}'(G)\leq \chi_{\rm irr}'(F) + \chi_{\rm irr}'(G') \leq 3$. Thus, we may assume that $v$ exists. Notice that $G'$ might consist of several connected components, but every component not containing $v$ is locally irregular. Let $H$ denote the connected component of $G'$ containing $v$.

If there exists no cycle through $v$ in $H$, then all edges incident with $v$ are cut-edges. Let $e$ be an edge incident with $v$, and let $H_1$ and $H_2$ denote the two connected components of $H-e$. We may assume that $H_1$ contains $v$. Notice that the degree of $v$ in $H_1$ and in $H_2+e$ is odd, while the degrees of its neighbours are even. It follows that both $H_1$ and $H_2+e$ are locally irregular and hence $$\chi_{\rm irr}'(G)\leq \chi_{\rm irr}'(F) + \chi_{\rm irr}'(H_1) + \chi_{\rm irr}'(H_2+e)\leq 4\,.$$

Thus, we may assume that there exists a cycle going through $v$. Let $C$ be a cycle through $v$ of shortest length and set $H'=H-E(C)$. Since the parities of the degrees remain unchanged, the vertex $v$ is still the only vertex in $B$ that could have positive even degree in $H'$, while all vertices in $A$ have even degree.
By Lemma~\ref{lemma-deletePath}, there exists a path $P$ in $H'$ starting in $v$ such that $H'-E(P)$ is locally irregular. Now $\chi_{\rm irr}'(C\cup P) \leq 4$ by Lemma~\ref{lemma-path+cycle} and we have
$$\chi_{\rm irr}'(G)\leq \chi_{\rm irr}'(F) + \chi_{\rm irr}'(H'-E(P)) + \chi_{\rm irr}'(C\cup P)\leq 2+1+4=7\,.$$
\end{proof}

We are now ready for the main result of this section.

\begin{theorem}\label{thm-bipeven}
If $G$ is a connected bipartite graph of even size, then $\chi_{\rm irr}'(G)\leq 9$.
\end{theorem}
\begin{proof}
By Lemma~\ref{lemma-makeAeven}, we can delete a balanced forest $F$ of $G$ so that the degrees in $A$ in the resulting graph $G'$ are even. By Lemma~\ref{lemma-balanced} we have $\chi_{\rm irr}'(F)\leq 2$, and $\chi_{\rm irr}'(G')\leq 7$ follows from Lemma~\ref{lemma-Aiseven}. Thus $\chi_{\rm irr}'(G)\leq \chi_{\rm irr}'(F) + \chi_{\rm irr}'(G') \leq 2+7 =9$.
\end{proof}

\begin{corollary} \label{corollary:bipodd}
If $G$ is a connected bipartite graph and not an odd length path, then $\chi_{\rm irr}'(G)\leq 10$.
\end{corollary}
\begin{proof}
Since paths of odd lengths are the only exceptional bipartite graphs, this follows immediately from Theorems~\ref{thm-odd} and~\ref{thm-bipeven}.
\end{proof}

\section{Locally irregular decompositions of degenerate graphs} \label{section:degenerate}

Here we apply the result from the previous section by decomposing degenerate graphs into bipartite graphs of even size. 
We show that every connected $d$-degenerate graph of even size can be decomposed into at most $\lceil \log_2 (d+1) \rceil + 1$ bipartite graphs
whose components all have even size. The proof makes repeated use of the following easy lemma.

\begin{lemma}\label{lemma-observ}
If $G$ is a graph with a vertex $v$ such that $G-v$ is bipartite, then there exists a set $E$ of at most $\lfloor \frac{d(v)}{2} \rfloor$ edges incident with $v$ such that $G-E$ is bipartite.
\end{lemma}
\begin{proof}
Since $G-v$ is bipartite, there exists a partition class containing at most $\lfloor \frac{d(v)}{2} \rfloor$ neighbours of $v$. Deleting all edges in $G$ from $v$ to these vertices results in a bipartite graph.
\end{proof}

\begin{lemma}\label{lemma-indstep}
Let $d$ be an even natural number, $\ell\geq \lceil\log_2 d\rceil +1$, and $v$ a vertex of degree $d$ in a graph $G$. If $G-v$ is the edge-disjoint union of $\ell$ bipartite graphs in which every component has even size, then so is $G$. 
\end{lemma}
\begin{proof}
Notice that it suffices to prove the statement for $\ell = \lceil\log_2 d\rceil +1$.
We use induction on $d$. In the case $d=2$ we colour $G-v$ with colours 1 and 2 so that the monochromatic connected components are bipartite subgraphs of even size. Let $u_1, u_2$ be the neighbours of $v$ in $G$. If $u_1$ and $u_2$ are not connected by an odd length path in colour 1, then colouring both $vu_1$ and $vu_2$ with colour 1 will keep all monochromatic components bipartite and of even size. Thus, we may assume that $u_1$ and $u_2$ are connected by a monochromatic path of odd length in each colour. Let $P=v_0v_1\ldots v_k$ be a monochromatic path from $u_1$ to $u_2$ in colour 2, so $v_0=u_1$ and $v_k=u_2$. Suppose that for every $i\in \{0,\ldots ,k-1\}$ there exists an even length path in colour 1 from $v_i$ to $v_{i+1}$. By concatenating them, we get a walk of even length from $v_0$ to $v_k$. Since there is also a path of odd length joining $v_0$ and $v_k$ in colour 1, this contradicts the assumption that the subgraph in colour 1 is bipartite.
Thus, there exists $i\in \{0,\ldots ,k-1\}$ for which there is no even length path in colour $1$ from $v_i$ to $v_{i+1}$. Choose $i$ minimal with this property. We change the colour of $v_iv_{i+1}$ to colour 1. By the choice of $i$, all monochromatic components in colour 1 are still bipartite. Now there exists precisely one monochromatic component of odd size in each colour. Notice that the monochromatic component of odd size in colour 1 is incident with both $u_1$ and $u_2$, while the one in colour 2 is incident with at least one of $u_1$ and $u_2$. Thus, we can colour one of the edges at $v$ with colour 2 so that all monochromatic components in colour 2 are bipartite and of even size. Colouring the other edge at $v$ with colour 1 yields the desired decomposition.

Now suppose $d\geq 4$ and that the statement is true for all smaller even numbers. Set $d'=\frac{d}{2}$ if $d$ is divisible by 4, and $d'=\frac{d}{2}+1$ otherwise. Notice that $d'$ is even and $\lceil\log_2 d\rceil = \lceil\log_2 d'\rceil +1$. Let $\mathcal{H}$ be the collection of $\lceil\log_2 d\rceil +1$ bipartite graphs in $G-v$ with even component sizes. Choose $H\in \mathcal{H}$ and denote by $G_H$ the graph we get by adding $v$ and all its incident edges to $H$. By Lemma~\ref{lemma-observ}, there exists a set $E$ of $d'$ edges incident with $v$ such that $G_H-E$ is bipartite. Since $d-d'$ is even, all connected components of $G_H-E$ have even size. We add the edges in $E$ to the union of the graphs in $\mathcal{H}\setminus \{H\}$ to obtain a graph $G'$. By the induction hypothesis, we can decompose $G'$ into $\lceil\log_2 d'\rceil +1$ bipartite graphs where every component has even size. Together with $G_H-E$, this is a collection of $\lceil\log_2 d'\rceil +2 = \lceil\log_2 d\rceil +1$ such graphs.
\end{proof}

Notice that the bound $\lceil\log_2 d\rceil +1$ can in general not be descreased by more than 1. The complete graph $K_{d+1}$ is $d$-degenerate and at least $\lceil \log_2 (d+1) \rceil$ bipartite graphs are needed to decompose it. Moreover, we might need more bipartite graphs to achieve that all components have even size. For example, the complete graph $K_4$ can be decomposed into two bipartite graphs, but three bipartite graphs are necessary to achieve even component sizes.

\begin{theorem}\label{thm-deg}
Let $d\geq 1$ be a natural number. If $G$ is a $d$-degenerate graph in which every connected component has even size, then $G$ can be decomposed into $\lceil \log_2 (d+1) \rceil + 1$ bipartite graphs in which all connected components have even size.
\end{theorem}

\begin{proof}
Suppose not, and let $G$ be a smallest counterexample. Clearly $G$ is connected.

\begin{claim*}
If $v$ is a cutvertex of $G$, then $v$ is adjacent to precisely one vertex $u$ of degree 1 and $G-u-v$ is connected.
\end{claim*}

To prove the claim, suppose there exists a $1$-separation $\{V_1,V_2\}$ of $G$ with $V_1\cap V_2 =\{v\}$ and $|V_1|$, $|V_2|\geq 3$.
If $G[V_1]$ and $G[V_2]$ have even size, then we can decompose $G[V_1]$ and $G[V_2]$ by induction. If $G[V_1]$ and $G[V_2]$ have odd size, then we construct two new graphs $H_1$ and $H_2$ by adding a new vertex $v_i$ to $G[V_i]$ together with the single edge $vv_i$. Since $|V_1|$, $|V_2|\geq 3$, both $H_1$ and $H_2$ are smaller than $G$ so we can decompose them by induction. We think of the decomposition as an edge-colouring, and we permute colours so that the edges $vv_i$ receive the same colour in both subgraphs. This corresponds to a colouring of $G$ in which every monochromatic component is bipartite and of even size.
This proves the claim.

\vspace{3mm}

In particular, every vertex is adjacent to at most one vertex of degree 1.
Among all vertices of degree greater than 1, let $v$ be one of minimal degree. Since $G$ is $d$-degenerate, we have $d(v)\leq d+1$. Suppose first that $d(v)$ is even. Since $G$ is a smallest counterexample, we can decompose $G-v$ into $\lceil \log_2 (d+1) \rceil + 1$ bipartite graphs in which all connected components have even size. By Lemma~\ref{lemma-indstep}, this gives rise to the desired decomposition of $G$.

We may thus assume that $d(v)$ is odd. Set $d'=\frac{1}{2}(d(v)-1)$ if $d(v)$ is congruent to 1 modulo 4, and $d'=\frac{1}{2}(d(v)+1)$ otherwise.
Notice that $d'$ is even and $\lceil\log_2 (d+1)\rceil \geq \lceil\log_2 d'\rceil +1$.
Let $u$ be a neighbour of $v$ of degree greater than $1$. If $G-v$ has an isolated vertex, then we let $w$ denote that vertex. Otherwise we add an isolated vertex $w$. The graph $G-v+uw$ has even size, so we can decompose it as in the previous case. This gives us a decomposition of $G-v$ into $\lceil \log_2 (d+1) \rceil + 1$ bipartite graphs in which all connected components are of even size, apart from one component of odd size which is incident with $u$. Let $H$ be the bipartite subgraph of odd size, and let $H_o$ be the connected component of odd size.
Let $G_H$ be the graph we get by adding $v$ and all its incident edges to $H$.
By Lemma~\ref{lemma-observ}, there exists a set $E$ of precisely $d'$ edges incident with $v$ such that $G_H-E$ is bipartite. We may assume that $E$ does not contain all edges that are incident with $H_o$. Since $d(v)-d'$ is odd, all connected components of $G_H-E$ have even size. We add the edges in $E$ to $G-v-E(H)$ to obtain a graph $G'$. Notice that $G-v-E(H)$ is the union of $\lceil\log_2 (d+1)\rceil$ bipartite graphs with components of even size.
By Lemma~\ref{lemma-indstep}, we can decompose $G'$ into $\lceil\log_2 (d+1) \rceil $ bipartite graphs where every component has even size. Together with $G_H-E$, this is a collection of $\lceil\log_2 (d+1)\rceil +1$ such graphs.
\end{proof}

Now we can use our result on bipartite graphs to get an upper bound on the irregular chromatic index of $d$-degenerate graphs.

\begin{corollary}\label{cor-deg}
If $G$ is a connected $d$-degenerate graph of even size, then
$$\chi'_{\rm irr}(G) \leq 9(\lceil\log_2 (d+1)\rceil +1)\,.$$
\end{corollary}

\begin{proof}
This follows immediately from Theorems~\ref{thm-bipeven} and~\ref{thm-deg}.
\end{proof}

To get a constant upper bound for decomposable graphs in general, we combine Corollary~\ref{cor-deg} with Przyby{\l}o's result on graphs with large minimum degree. For this purpose, we need the following lemma.

\begin{lemma}\label{lemma-deg+min}
Let $d$ be a natural number. If $G$ is a connected graph of even size, then $G$ can be decomposed into two graphs $D$ and $H$ such that $D$ is $2d$-degenerate, every connected component of $D$ has even size, and the minimum degree of $H$ is at least $d$.
\end{lemma}

\begin{proof}
Starting from $D=\emptyset$ and $H=G$, we remove vertices of degree at most $2d$ from $H$ and add them to $D$. Once this process stops, the graph $D$ is $2d$-degenerate and $H$ has minimum degree at least $2d+1$. Every connected component $C$ of $D$ with odd size intersects $H$; let $v(C)$ be a vertex in the intersection. Notice that $v(C)\neq v(C')$ for different connected components $C$ and $C'$ of $D$. We choose an almost-balanced orientation of $H$, i.e. an orientation where the out-degree and in-degree at every vertex differ by at most 1. For each connected component $C$ of odd size, we choose an out-edge $e(C)$ at $v(C)$ in $H$. We remove $e(C)$ from $H$ and add it to $D$. Since every vertex in $H$ might lose all of its in-edges but at most one out-edge, the minimum degree in $H$ remains at least $d$. The edges we add to $D$ in this step induce a 2-degenerate subgraph, so $D$ will still be $2d$-degenerate. Moreover, every connected component of odd size gains an edge and possibly gets joined to other connected components of even size. In any case, all connected components of $D$ now have even size.
\end{proof}

Now we are ready for the proof of our main result.

\begin{theorem} \label{main-theorem}
If $G$ is a decomposable graph, then $\chi'_{\rm irr}(G) \leq 328\,.$
\end{theorem}

\begin{proof}
By Theorem~\ref{thm-odd} it suffices to show that $\chi'_{\rm irr}(G) \leq 327$ holds for connected graphs $G$ of even size. By Lemma~\ref{lemma-deg+min}, we can decompose $G$ into two graphs $D$ and $H$ such that $D$ is $(2\cdot 10^{10})$-degenerate, every connected component of $D$ has even size, and the minimum degree of $H$ is at least $10^{10}$. By Theorem~\ref{theorem:przybylo}, we have $\chi_{\rm irr}'(H)\leq 3$ and by Corollary~\ref{cor-deg} we have $$\chi_{\rm irr}'(D)\leq 9(\lceil\log_2 (2\cdot 10^{10}+1)\rceil +1) = 324\,.$$
Hence, $\chi_{\rm irr}'(G)\leq \chi_{\rm irr}'(H) + \chi_{\rm irr}'(D)\leq 3 + 324 = 327$.
\end{proof}

\section{Decomposing highly edge-connected bipartite graphs} \label{section:highly-connected}

Our bound on the irregular chromatic index of decomposable graphs in Theorem~\ref{main-theorem},
depends partly on the irregular chromatic index of bipartite graphs.
In particular, decreasing our bound in Theorem~\ref{thm-bipeven} from~$9$ down to~$3$ (as suggested by Conjecture~\ref{conj-3})
would already yield an improvement on the constant in Theorem~\ref{main-theorem}.
The restriction of Conjecture~\ref{conj-3} to bipartite graphs, however,
appears to be a surprisingly non-trivial problem.

Another interesting question concerning bipartite graphs and locally irregular decompositions
is about whether the family of bipartite graphs with irregular chromatic index at most~$2$
admits an ``easy'' characterisation.
It is legitimate to raise this question, as trees admit such a characterisation, see~\cite{BBS15}.
We also note that a similar study of bipartite graphs $G$ satisfying $\chi'_\Sigma(G) \leq 2$
was recently conducted by Thomassen, Wu, and Zhang~\cite{TWZ},
resulting in such a characterisation.
So far all bipartite graphs $G$ with irregular chromatic index 3 we know have minimum degree 1 or 2. It could be the case that minimum degree 3 already suffices to push the irregular chromatic index down to 2 for bipartite graphs. Notice that $d$-regular bipartite graphs $G$ with $d\geq 3$ do indeed satisfy $\chi'_{\rm irr}(G) = 2$, as was shown in~\cite{BBPW15}.

\begin{question} \label{conjecture:bip-degree-3}
Does there exist a bipartite graph $G$ with minimum degree at least 3 and $\chi'_{\rm irr}(G) > 2$?
\end{question}

In this section we prove that $16$-edge-connected bipartite graphs have irregular chromatic index at most~$2$.
The main tool is the following result on factors modulo~$k$ in bipartite graphs, due to Thomassen~\cite{Tho14}. That result was based on the proof of the weak version of Jaeger's Circular Flow Conjecture in~\cite{Tho12}. The quadratic bound on the edge-connectivity in~\cite{Tho12} was improved to a linear bound in \cite{LTZW}.

\begin{theorem}[\cite{Tho14}] \label{theorem:factor}
Let $k$ be a natural number, and let $G$ be a $(3k-2)$-edge-connected bipartite graph with partition classes $A$ and $B$. 
If $f:V(G)\rightarrow \mathbb{Z}$ is a function satisfying 
$$ \sum_{v\in A} f(v) \equiv \sum_{v\in B} f(v) \pmod k \,,$$
then $G$ has a spanning subgraph $H$ with $d(v, H) \equiv f(v) \pmod k$ for every $v\in V(G)$.
\end{theorem}

Theorem 1 in \cite{Tho12} assumes edge-connectivity $3k-3$. But this holds only for $k$ odd. For $k$ even it should be $3k-2$, see page 11 in \cite{LTZW}, and below we shall apply
Theorem \ref{theorem:factor} for $k=6$.

\begin{theorem}
For every $16$-edge-connected bipartite graph $G$, we have $\chi'_{\rm irr}(G) \leq 2$.
\end{theorem}

\begin{proof}
Let $A$ and $B$ denote the partition classes of $G$. Since $G$ is $16$-edge-connected, we have $|A|,|B|\geq 16$.

Suppose first that every vertex has at most one non-neighbour in the other partition class. Let us assume that $|A|\geq |B|$. If $|A|-|B|\geq 2$, then $G$ is already locally irregular. If $|A|-|B|\leq 1$, then let $H$ be a subgraph of $G$ consisting of two vertices in $B$ and all edges incident with one of these vertices. Clearly $H$ is locally irregular. Let $H'$ denote the subgraph $G-E(H)$. For $v\in A$, we have $d(v,H')\leq |B|-2 \leq |A|-2.$
For $v\in B$ we have $d(v,H')\geq |A|-1$, so $H'$ is locally irregular and $\chi'_{\rm irr}(G) \leq 2$.

Now suppose that $G$ contains a vertex $u$ with at least two non-neighbours in the other partition class. We may assume $u\in A$. We denote by $A_0$ and $B_0$ the subsets of $A$ and $B$ consisting of the vertices of even degree, and by $A_1$ and $B_1$, respectively, the subsets consisting of the vertices of odd degree.
For any function $g:V(G)\rightarrow \mathbb{Z}$, we define $\sigma_A(g) = \sum_{v \in A} g(v)$ and $\sigma_B(g) = \sum_{v \in B} g(v)$.
If $G$ is $(3k-2)$-edge-connected and $\sigma_A(g) \equiv \sigma_B(g)$ (mod $k$), then, by Theorem~\ref{theorem:factor}, there exists a subgraph $H$ of $G$ with $d(v,H)\equiv g(v)$ (mod $k$) for all $v\in V(G)$. We apply this for $k=6$.
We assign to each vertex $v\in V(G)$ an integer $f(v)$ in the following way:
\begin{itemize}
	\item for every $v \in A_0$, we set $f(v)=0$;
	\item for every $v \in A_1$, we set $f(v)=1$;
	\item for every $v \in B_0$, we set $f(v)=3$;
	\item for every $v \in B_1$, we set $f(v)=2$.
\end{itemize}
If $\sigma_A(f) \equiv \sigma_B(f)$ (mod $6$), then we apply Theorem~\ref{theorem:factor} to find a subgraph $H$ of $G$ with $d(v,H)\equiv f(v)$ (mod $6$) for all $v\in V(G)$. Since the degrees of the vertices in $A$ have different residues modulo $6$ than the degrees of the vertices in $B$, the graph $H$ is locally irregular. Let $H'$ denote the subgraph $G-E(H)$. The vertices in $A$ in $H'$ have even degrees, while the vertices in $B$ have odd degrees. Thus, also $H'$ is locally irregular and $\chi'_{\rm irr}(G) \leq 2$.

We may therefore assume $\sigma_A(f) \not\equiv \sigma_B(f)$ (mod $6$). First, suppose $\sigma_A(f) \equiv \sigma_B(f)$ (mod $2$). Then $\sigma_A(f) \equiv \sigma_B(f) + 2$ or $\sigma_A(f) \equiv \sigma_B(f) + 4$. Let $x$ and $y$ be two different vertices in $B$. We define two new functions $f_1$ and $f_2$ by setting $f_1(x)=f_2(x)=f(x)+2$, $f_2(y)=f(y)+2$ and setting $f_1(v)=f(v)$ for all $v\in V(G)\setminus\{x\}$ and $f_2(v)=f(v)$ for all $v\in V(G)\setminus\{x,y\}$. Now $\sigma_A(f_1)=\sigma_A(f_2)=\sigma_A(f)$ and $\sigma_B(f_2)= \sigma_B(f_1) + 2= \sigma_B(f)+4$. Thus, one of the functions $f_1$ or $f_2$ satisfies the condition in Theorem~\ref{theorem:factor} and the same argument as above yields a decomposition into two locally irregular subgraphs.

Finally, suppose $\sigma_A(f) \not\equiv \sigma_B(f)$ (mod $2$). We define a new function $g$ by setting $g(u)=1-f(u)$, where $u$ is the special vertex with at least two non-neighbours in $B$. We set $g(v)=f(v)$ for all $v\in A$ and all $v\in B$ that are non-neighbours of $u$. If $v\in B$ is a neighbour of $u$ with $d(v)-f(v)\not\equiv d(u)-g(u)$ (mod $6$), then we also set $g(v)=f(v)$. In the case that $v\in B$ is a neighbour of $u$ with $d(v)-f(v)\equiv d(u)-g(u)$ (mod $6$), we set $g(v)=f(v)+2$.
Now we have $\sigma_A(g) \equiv \sigma_A(f)+1$ and $\sigma_B(g)\equiv \sigma_B(f)$, so $\sigma_A(g) \equiv \sigma_B(g)$ (mod $2$). 

If $\sigma_A(g) \equiv \sigma_B(g)$ (mod $6$), then we use Theorem~\ref{theorem:factor} for the function $g$. Let $H$ be the subgraph of $G$ with $d(v,H)\equiv g(v)$ (mod $6$) for all $v\in V(G)$. Notice that $g(v)\in \{0,1\}$ for $v\in A$ and $g(v)\in \{2,3,4,5\}$ for $v\in B$, so $H$ is locally irregular. Let $H'$ denote the subgraph $G-E(H)$. Notice that in $H'$ all vertices in $B$ have odd degree, while $u$ is the only vertex in $A$ of odd degree. Notice that for every neighbour $v$ of $u$ we have $d(v,G)-g(v) \not\equiv d(u,G)-g(u)$ (mod $6$) by definition of $g$, so $u$ is not adjacent to a vertex of the same degree in $H'$. Thus, also $H'$ is locally irregular.

If $\sigma_A(g) \not\equiv \sigma_B(g)$ (mod $6$), then we define functions $g_1$ and $g_2$ as before by adding 2 to the values of one or two vertices in $B$, and use Theorem~\ref{theorem:factor} for the function $g_1$ or $g_2$. This time we only choose vertices that are non-neighbours of $u$, to make sure the value at a vertex never increases by more than 2 compared to its original $f$-value. As before, this yields a decomposition of $G$ into two locally irregular subgraphs.
\end{proof}

\end{document}